\documentclass[11pt]{amsart}
\usepackage{texmaX,semmaX,semtkzX}
\usepackage{makecell}
\usepackage{multirow}
\newcommand{\Mod}[1]{\ \mathrm{mod}\ #1}

\begin{document}

\title[On the order of the Tate--Shafarevich group modulo squares]{A note on the order of the Tate--Shafarevich group modulo squares}
\author{Alexandros Konstantinou}
\address{University College London, London WC1H 0AY, UK}
\email{alexandros.konstantinou.16@ucl.ac.uk}
\begin{abstract}
We investigate the order of the Tate--Shafarevich group of abelian varieties modulo rational squares. Our main result shows that every square-free natural number appears as the non square-free part of the Tate--Shafarevich group of some abelian variety, thereby validating a conjecture of W. Stein. 
\end{abstract}
\maketitle
\setcounter{tocdepth}{1}
\tableofcontents
\section{Introduction}
In the following note, we study the order of the Tate--Shafarevich group $\Sha(A/K)$ of an abelian variety $A$ defined over a number field $K$. Its non-zero elements correspond to principal homogeneous spaces for $A/K$ with a rational point over all completions of $K$ though lacking a point in $K$ itself. Despite its central role in the study of rational points, several foundational questions regarding $\Sha(A/K)$ remain open. These include the Tate--Shafarevich conjecture which stipulates its finiteness. In addition, a conjecture of Stein focuses on the possible values that the order of the Tate--Shafarevich group can take (if finite) modulo rational squares. In the following note, we validate Stein's conjecture, which can be phrased as follows.
\begin{theorem}[=Corollary \ref{Sha-n-square}]
For every positive square-free integer $n$, there exists an abelian variety $A/\mathbb{Q}$ with finite Tate--Shafarevich group of order $nm^2$ for some integer $m \geq 1$.
\end{theorem}
\subsection{Previous work} In the case of principally polarised abelian varieties, the Cassels--Tate pairing, \[\langle,\rangle: \Sha(A/K) \times \Sha(A/K) \to \mathbb{Q}/\mathbb{Z},\] lays down significant constraints on the structure of the Tate--Shafarevich group. In particular, for elliptic curves, this is known to be alternating, and therefore the order of the Tate--Shafarevich group of an elliptic curve must be a square, if finite. 

The work of Flach \cite{Flach} establishes the (slightly weaker) skew symmetric condition $\langle x, y \rangle = - \langle y , x\rangle$ for principally polarised abelian varieties. From this, one can \textit{only} deduce that the Tate--Shafarevich group is either a square or twice a square. Building upon this, Poonen and Stoll present a criterion for deducing if the Tate--Shafarevich group (provided it's finite) is of square order purely based on the vanishing of $\langle c, c \rangle$ for an element $c \in \Sha(A/K)[2]$ canonically associated to $A$. They then use this criterion to present the first example of an abelian variety for which the order of its Tate--Shafarevich group is twice a square.

 Subsequently, Stein \cite{Stein} presents the first examples of abelian varieties defined over $\mathbb{Q}$ in which the order of their Tate--Shafarevich groups satisfy $\left |\Sha(A/\mathbb{Q}) \right | = p m^2$ for $p < 25000$ an odd prime. His construction relies on the existence of an abelian variety defined over $\mathbb{Q}$ which becomes isomorphic to a product $E^{p-1}$ of elliptic curves over a cyclic extension. 
Following this, Keil's work \cite{Keil} presents examples of abelian surfaces, given as quotients of a product of two elliptic curves, for which $\left | \Sha \right | = p m^2$ for $p \leq 13$ an odd prime.

\subsection{Acknowledgements} I would like to extend my thanks to Vladimir Dokchitser for his guidance and many constructive discussions during the course of this work. In addition, I would like to thank Adam Morgan for many helpful suggestions on a draft version of this paper and for many useful discussions that came before that. 

\section{Expression for $\Sha$ mod squares} \label{sec_sha_mod_squares}
Our proof relies on breaking down the Weil restriction of scalars of an abelian variety up to isogeny followed by an application of a formula of Cassels--Tate. This allows us to express the size of the Tate--Shafarevich group modulo squares in terms of Birch--Swinnerton-Dyer constants.

\subsection{Cassels--Tate formula} In the following subsection, we briefly recall a result of Cassels and Tate which proves invariance of the Birch--Swinnerton-Dyer conjecture under isogeny. We first fix some notation. 
\begin{notation} We let $K$ be a number field. Given a $K$-rational isogeny $\phi: A_1 \to A_2$, we write 
\begin{equation*}\begin{split} 
Q(\phi) =&|\textup{coker}(\phi: A_1(K)/A_1(K)_{\textup{tors}} \to A_2(K)/A_2(K)_{\textup{tors}})| \\ \times &|\textup{ker}(\phi: \Sha(A_1/K)_{\textup{div}} \to \Sha(A_2/K)_{\textup{div}})|, \end{split}\end{equation*}
where $\Sha_{\textup{div}}$ denotes the divisible part of $\Sha$. 

For an abelian variety $A/K$ with a fixed non-zero global exterior form $\omega$, we denote the Birch--Swinnerton-Dyer periods by \[\Omega_{\mathbb{C}}(A,\omega)=2^{\textup{dim}A} \int_{A(\mathbb{C})}  \left |\omega \wedge \overline{\omega} \right |, \ \ \  \ \ \Omega_{\mathbb{R}}(A,\omega)=\int_{A(\mathbb{R})}  \left | \omega\right |.\]
We write $\Omega(A/K,\omega) = \prod_{v \mid \infty} \Omega_{K_v}(A,\omega)$ where $K_v = \mathbb{R}$ (resp. $\mathbb{C}$) if $v$ is real (resp. complex). In addition, for a non-archimedean place $v$, we write $w_{A,v}^{\textup{min}}$ for a N\'{e}ron minimal exterior form on $A$, $c_v(A/K)$ for its Tamagawa number and $C(A/K,\omega)=\prod_{v \nmid \infty}c_v(A/K)|{\omega}/{\omega_{A,v}^{\textup{min}}} |_v $. We write $\Sha_0 = \Sha/\Sha_{\textup{div}}$ and $\square$ to denote the square of a rational number.

\end{notation}
The following theorem is a version of the Cassels--Tate formula as found in {\cite[Theorem 4.3]{BSD-mod-squares}}.
\begin{theorem}[\textup{Cassels--Tate}] \label{Cassels-tate}
Let $\phi:A_1 \to A_2$ be a $K$-rational isogeny. Then, for any non-zero global exterior forms $\omega_1$ and $\omega_2$ on $A_1$ and $A_2$,
\begingroup
\fontsize{9.5pt}{9pt}
\begin{equation*}  \frac{Q(\phi^{\vee})}{Q(\phi)}=\frac{|A_2(K)_{\textup{tors}} |}{|A_1(K)_{\textup{tors}} |}\frac{|A_2^{\vee}(K)_{\textup{tors}} |}{|A_1^{\vee}(K)_{\textup{tors}} |} \frac{C(A_1/K,\omega_1)}{C(A_2/K,\omega_2)} \frac{\Omega(A_1/K,\omega_1)}{\Omega(A_2/K,\omega_2)} \prod_{p \mid \textup{deg}(\phi)} \frac{ | \Sha_0(A_1/K)[p^{\infty}] |}{ | \Sha_0(A_2/K)[p^{\infty}] |}. \end{equation*}
\endgroup
\end{theorem}

Given a field extension $L/K$, we write $\phi_L$ for the base change of $\phi$ to $L$. The following proposition shows that for carefully chosen $L$ and working modulo squares, the Cassels--Tate formula simplifies.
\begin{lemma} \label{Lem:CT_quadratic_extension_general}
Let $L/K$ be a quadratic extension such that all bad primes of $A_1/K$ and $A_2/K$ and all places $v \mid \textup{deg}(\phi)$ split in $L$. In addition, suppose that $\textup{deg}(\phi)$ is a square. Then,
\[ \frac{Q(\phi_L^{\vee})}{Q(\phi_L)}  = \frac{|A_2(L)_{\textup{tors}} |}{|A_1(L)_{\textup{tors}} |}\frac{|A_2^{\vee}(L)_{\textup{tors}} |}{|A_1^{\vee}(L)_{\textup{tors}} |} \prod_{p \mid \textup{deg}(\phi)} \frac{ | \Sha_0(A_1/K)[p^{\infty}] |}{ | \Sha_0(A_2/K)[p^{\infty}] |}  \Mod \square. \]
\end{lemma}
\begin{proof}
We note that $C(A_1/L,\omega_1) \Omega(A_1/L,\omega_1)$ is independent of $\omega_1$. By our assumptions on $L/K$, the only terms in this product we have to consider are complex places of $L$ which lie above real places of $K$ and non-archimedean places $w \nmid \textup{deg}(\phi)$. As in the proof of \cite[Theorem 7.3]{ADT}, when $w$ is complex, we have  ${\Omega_{\mathbb{C}}(A_1,\phi^{*}\omega_2)}/{\Omega_{\mathbb{C}}(A_2,\omega_2)} = \textup{deg}(\phi) = \square.$ For a non-archimedean place $w \nmid \textup{deg}(\phi)$, $|{(\phi^{*}\omega_2/\omega_{A_1,w}^{\textup{min}})}\big{/}{(\omega_2/\omega_{A_2,w}^{\textup{min}})} |_w =1$ (see \cite[pp. 12]{Schaefer}). To conclude, $\frac{C(A_1/L,\omega_1)}{C(A_2/L,\omega_2)}  \frac{\Omega(A_1/L,\omega_1)}{\Omega(A_2/L,\omega_2)}=\square$ giving the required result. \qedhere
\end{proof}

\subsection{A decomposition for the Weil-restriction of scalars} \label{subsec_setup} In what follows, we let $A/K$ be a principally polarised abelian variety defined over a number field $K$. We let $p$ be an odd prime, $F/K$ a cyclic extension of degree $p$ and  $\textup{Res}_{K}^{F}A$ the restriction of scalars from $F$ to $K$. For brevity, we write $B=\textup{Res}_{K}^{F}A$. This is an $F/K$-twist of $A^{p}$, and the product polarisation on $A^{p}$ descends to a principal polarisation on $B/K$. Then, there are $K$-morphisms 
\[i: A \to B, \ \ \  \ \ \textup{Tr}:B \to A. \]
These are the inclusion and the trace map respectively. 

After identifying the base change $B_F=A^p$ as principally polarised abelian varieties, the base changed morphisms $i_F$ and $\textup{Tr}_F$ coincide with the diagonal inclusion and the summation map respectively. In addition, writing $i^{\vee}: B \to A$ for the dual of $i$ (having suppressed the canonical principal polarisations on both sides), we have $i^{\vee}=\textup{Tr.}$ Finally, we write $R$ to denote the kernel of the trace map. This is an abelian variety defined over $K$ and is an $F/K$-twist of $A^{p-1}$ (see \cite[Proposition 2.4]{Stein}).

\begin{lemma} \label{kernel_of_isogeny} Let $\phi: R \times A \to B$ be given by $(x,y) \mapsto x + i(y)$. Then, $\phi$ is a $K$-rational isogeny with  $\textup{ker}(\phi) \cong A[p]$.
\end{lemma}
\begin{proof}
It suffices to show $\textup{ker}(\phi) \cong A[p]$. After base change to $\overline{K}$, we identify $R_{\overline{K}} = A^{p-1}$ and $B_{\overline{K}}=A^{p}$. Then, on $\overline{K}$-points $\phi$ is given by $(x_1,\dots,x_{p-1},y) \mapsto (y+x_1,\dots,y+x_{p-1},y-\sum_{i=1}^{p-1}x_{i})$ . Thus, restricting the projection $R \times A \to A$ to $\textup{ker}(\phi)$ allows us to deduce that $\textup{ker}(\phi) \cong A[p]$.
\end{proof}

\subsection{An expression for $\Sha(R/L)$ modulo squares}
In this subsection, we relate  $| \Sha(R/L)[p^{\infty}]| \Mod \square$ to the torsion subgroup of $R$ and the rank of $A/L$.

\begin{lemma} \label{Lem:size_of_shas} Let $L/K$ be a number field extension such that $L\otimes_{K}F$ is a field. Then, the following hold.
\begin{enumerate}
\item For all odd primes $q \neq p$, $\left | \Sha_{0}(R/L)[q^{\infty}]\right | = \square$.
\item If, in addition, $A$ is an elliptic curve, $\left | \Sha_{0}(R/L)[2^{\infty}]\right | = \square$.
\end{enumerate}
\end{lemma}
\begin{proof}
By Lemma \ref{kernel_of_isogeny},  $\textup{deg}(\phi)=p^{2\textup{dim}(A)}$. When $q \neq p$, this isogeny induces an isomorphism $\Sha_{0}(A/L)[q^{\infty}] \times \Sha_{0}(R/L)[q^{\infty}] \stackrel{\sim}\to \Sha_{0}(B/L)[q^{\infty}]$. Since $A$ and $B$ are principally polarised, \cite[Theorem 8]{poonen1999cassels} asserts that $\Sha_{0}(A/L)[q^{\infty}]$ and $\Sha_{0}(B/L)[q^{\infty}]$ are both of square order for all odd $q$. This proves (1). Further, by \cite[Proposition A.5.2]{reductive-groups}, $B_L \cong \textup{Res}_{L}^{L\otimes_{K}F}A$, and by combining \cite[pp. 178(a)]{MR0330174} with Shapiro's Lemma, $\Sha_{0}(B/L)[2^{\infty}] \cong \Sha_{0}(A/L\otimes_{K}F)[2^{\infty}]$. Therefore, if $A$ is an elliptic curve, both $\Sha_{0}(A/L)[2^{\infty}]$ and $\Sha_{0}(B/L)[2^{\infty}]$ are of square order by \cite[Theorem 8, Corollary 12]{poonen1999cassels}. Claim (2) follows.
\end{proof}
In what follows, we write $\textup{rk}_{p}(A/K)$ for the $\mathbb{Z}_p$-corank of the $p^{\infty}$-Selmer group $\textup{Sel}_{p^{\infty}}(A/K):= \varinjlim_{n \geq 1} \textup{Sel}_{p^n}(A/K)$. This coincides with the Mordell--Weil rank of $A$ plus the multiplicity of $\mathbb{Q}_{p}/\mathbb{Z}_{p}$ in its Tate--Shafarevich group.
\begin{proposition} \label{Prop:CT_Weil_Restriction} Let $L/K$ be a quadratic extension such that all bad primes of $A/K$ and all primes dividing either $\textup{deg}(\phi)$ or the relative discriminant $\Delta(F/K)$ split in $L$. Then, the following hold.
\begin{enumerate}
\item Up to rational squares, 
\begin{equation*} 
\left | \Sha_{0}(R/L)[p^{\infty}]\right | \equiv \frac{Q(\phi_{L})}{Q{(}\phi^{\vee}_{L}{)}}  \frac{\left | R(L)[p^{\infty}]\right|} {\left | R^{\vee}(L)[p^{\infty}]\right|}  \Mod \square.
\end{equation*}
\item  If, in addition, $\textup{rk}_{p}(A/L)= \textup{rk}_{p}(A/L\otimes_{K}F)$, then
\begin{equation*} 
\left | \Sha_{0}(R/L)[p^{\infty}]\right | \equiv p^{\textup{rk}_{p}(A/L)} \frac{\left | R(L)[p^{\infty}]\right|} {\left | R^{\vee}(L)[p^{\infty}]\right|}  \Mod \square.
\end{equation*}
\end{enumerate}
\end{proposition}
\begin{proof}
Let $v$ be a place of bad reduction for $B/K$. By \cite[Proposition 1]{MR0330174}, this can only happen when $A/K$ has bad reduction at $v$ or when $v \mid \Delta(F/K)$. This shows that all bad primes of $B/K$ split in $L$.
Since $A/K$ is principally polarised, $\left |A(L)[p^{\infty}]\right|=\left | A^{\vee}(L)[p^{\infty}] \right |$ and $\left |\Sha_{0}(A/L)[p^{\infty}]\right|=\left | \Sha_{0}(A^{\vee}/L)[p^{\infty}] \right |$. The same holds for $B$. Part (1) follows by combining  Lemma \ref{Lem:CT_quadratic_extension_general} and Lemma \ref{kernel_of_isogeny}. For the second claim, it suffices to show that if $\textup{rk}_{p}(A/L)= \textup{rk}_{p}(A/L\otimes_{K}F)$, then $ {Q(\phi_{L})}/{Q{(}\phi^{\vee}_{L}{)}} =p^{\textup{rk}_{p}(A/L)}$. After suppressing the principal polarisations $A=A^{\vee}, B=B^{\vee}$, we view $\phi^{\vee}$ as a morphism $B \to A \times R^{\vee}$. Therefore, 
$\frac{Q(\phi_{L})}{Q(\phi^{\vee}_{L})} \equiv {Q(\phi^{\vee}_{L}\circ\phi_{L}: A \times R \to A \times R^{\vee}}) \Mod \square.$ Since the degree of $\phi^{\vee} \circ \phi$ is a power of $p$, it induces an isomorphism $\Sha(A\times R/L)_{\textup{div}}[q^{\infty}] \stackrel{\sim}{\to} \Sha(A\times {R}^{\vee}/L)_{\textup{div}}[q^{\infty}]$ for $q \neq p$. From this, we deduce $\textup{ker}(\phi_L^{\vee} \circ \phi_L \ | \ {\Sha}_{\textup{div}}) = \textup{ker}(\phi_L^{\vee} \circ \phi_L \ | \ {\Sha}_{\textup{div}}[p^{\infty}]) $. Since $B_L \cong \textup{Res}_{L}^{L\otimes_{K}F}A$ (see \cite[Proposition A.5.2]{reductive-groups}), then by combining \cite[pp. 178(a)]{MR0330174} with Shapiro's Lemma gives $\textup{Sel}_{p^{\infty}}(B/L) \cong \textup{Sel}_{p^{\infty}}(A/L\otimes_{K}F)$. By assumption, $\textup{rk}_{p}(A/L)=\textup{rk}_{p}(A/L\otimes_{K} F)$, from which we deduce $\textup{rk}_{p}(R/L)=\textup{rk}_{p}(R^{\vee}/L)=0$. As a result, ${Q(\phi^{\vee}_{L}\circ\phi_{L}})$ coincides with $Q(i_L^{\vee}\circ i_L)$, where we view $i_L^{\vee}\circ i_L$ as an endomorphism of $A_L$. Using the description of $i$ and $i^{\vee}$ from \S\ref{subsec_setup}, $i^{\vee} \circ i=[p]_{A}$. This gives the required result. \qedhere
 
\end{proof}

\section{$\Sha$ of non-square order}

We retain all the notation from \S \ref{sec_sha_mod_squares}. In this section,  we restrict to the case when $A=E$ is an elliptic curve, and so $R$ is an abelian variety of dimension $p-1$. In addition, we fix the following notation.

\begin{notation} \label{notation_fields_cylotomic}We write $K_{\infty}$ to denote the cyclotomic $\mathbb{Z}_{p}$-extension of $\mathbb{Q}$, and $K_n=\mathbb{Q}(\zeta_{p^n} + \zeta_{p^{n}}^{-1})$ for the $n^{\textup{th}}$ layer of this extension. For $d$ a positive, square-free integer, we let $L_n= K_n \cdot \mathbb{Q}(\sqrt{-d})$. \end{notation}

We first show that for $n$ sufficiently large, $\Sha(R/L_n)$ is finite for any choice of $d$. Further, having carefully selected $d $, we show that $|\Sha(R/L_n)| \equiv p \Mod \square$. We begin by proving a lemma involving torsion on $R$ and $R^{\vee}$.
\begin{lemma} \label{torsion_quotient_elliptic_curve}
Let $E$ be an elliptic curve defined over $\mathbb{Q}$. Then, for any $n \geq 0$ and any prime $p \geq 11$ both $R(L_n)[p^{\infty}]$ and $ R^{\vee}(L_n)[p^{\infty}]$ are trivial.
\end{lemma}
\begin{proof}
As  abelian varieties over $L_{n+1}$, both $R$ and $R^{\vee}$ are isomorphic to  $E^{p-1}.$ It therefore suffices to show that $E(L_{n+1})[p^{\infty}]$ is trivial for $ p \geq 11$. Let $E_{-d}$ be the quadratic twist of $E$ by $-d$. Then, there exists a map $E(L_{n+1}) \to E(K_{n+1}) \times E_{-d}(K_{n+1})$ whose kernel and cokernel are $2$-groups. Then, for any odd prime $p$, this induces an isomorphism  \[E(L_{n+1})[p^{\infty}] \stackrel{\sim}\to E(K_{n+1})[p^{\infty}] \times E_{-d}(K_{n+1})[p^{\infty}].\] By \cite[Theorem 1.1]{torsion_over_Zp},  $E(K_{\infty})[p^{\infty}]=E(\mathbb{Q})[p^{\infty}]$ (similarly for $E_{-d})$ for any $p\geq 5$. If $p\geq 11$, Mazur's theorem \cite[Theorem 2]{Mazur-isogeny-theorem} asserts that $E(\mathbb{Q})[p^{\infty}]$ and $E_{-d}(\mathbb{Q})[p^{\infty}]$ are both trivial. Therefore, $E(L_{n+1})[p^{\infty}]$ is trivial.
\end{proof}
 
The following result is a formal consequence of a theorem by Kato and Rohrlich. 

\begin{theorem}\label{RK}
Let $E$ be an elliptic curve defined over $\mathbb{Q}$. Then, there exists a positive integer $n$ such that
\begin{enumerate}
\item  $\textup{ord}_{s=1}L(E/L_{n},s) = \textup{ord}_{s=1}L(E/L_{n+1},s),$
\item $\Sha(R/L_{n})$ is finite, 
\item $\textup{rk}_{p}(E/L_{n+1}) = \textup{rk}_{p}(E/L_{n}).$
\end{enumerate}
\end{theorem}
\begin{proof}
As $E$ is defined over $\mathbb{Q}$, $\textup{Res}_{K_{n}}^{L_{n}}E $ is isogenous to $E \times E_{-d}$ giving an equality of $L$-functions
\begin{equation*} \label{decomposition_of_L_functions} L(E/L_{n},s) = L(E/K_{n},s)L(E_{-d}/K_{n},s), \end{equation*}
By Rohrlich \cite[pp. 409]{Rohrlich}, the sequence $\{\textup{ord}_{s=1}L(E/K_{n},s)\}_{n \geq 1}$ stabilises (similarly for $E_{-d}$), and so $\textup{ord}_{s=1}L(E/L_{n},s) =\textup{ord}_{s=1}L(E/L_{n+1},s)$ for $n$ sufficiently large. This proves (1). 
Part (2) is a consequence of a result by Kato \cite{Kato}. Following the notation introduced there, for $G$ an abelian group, $\tau \in \widehat{G}$, and a $G$-module $M$, we let $M^{(\tau)} :=\{ x \in M : I_{\tau} \cdot x =0\}$ where $I_{\tau} \subseteq \mathbb{Z}[G]$ denotes the kernel of the map $\mathbb{Z}[G] \to \mathbb{C}^{\times}$ induced by $\tau$. We apply this to $G=\textup{Gal}(L_{n+1}/\mathbb{Q}) \cong \mathbb{Z}/2p^{n+1}\mathbb{Z}$. For brevity, we write $H= \textup{Gal}(L_{n+1}/L_{n})\cong \mathbb{Z}/p\mathbb{Z}$. Further, the composition $i \circ \textup{Tr}$ from \S \ref{subsec_setup} is equal to $N_H:= \sum_{h \in H} h$ viewed as a $K$-endomorphism of $B$.  It follows that the composition $R \hookrightarrow{} B \stackrel{N_H}\to B$ is the zero map, which in turn gives a homomorphism $\Sha(R/L_n) \to \textup{ker}(N_H \ | \ \Sha(E/L_{n+1}))$. We claim this has finite kernel. To see this, we look at the short exact sequence of abelian varieties 
\[ 0 \to R \to B  \xrightarrow{\textup{Tr}} E \to 0.\] Taking Galois cohomology gives an exact sequence
\[ 0 \to E(L_n)/\textup{Norm}_{L_n}^{L_{n+1}}E(L_{n+1}) \to H^{1}(L_n,R) \to  H^{1}(L_n,B).\]
By the Mordell--Weil theorem $E(L_n)/pE(L_n)$ is finite, from which we deduce that $E(L_n)/\textup{Norm}_{L_n}^{L_{n+1}}E(L_{n+1})$ must be finite. Therefore, the induced map $\Sha(R/L_n) \to \textup{ker}(N_H \ | \ \Sha(E/L_{n+1}))$ has finite kernel. For a primitive character $ \tau \in \widehat{G}$, we have $N_H \in I_{\tau}$. Therefore, for such $\tau$, $\textup{ker}(N_H \ | \ \Sha(E/L_{n+1}))\subseteq \Sha(E/L_{n+1})^{(\tau)}$. By Artin formalism,
  \[L(E/L_{n+1},s) = L(E/L_{n},s)\prod_{\substack{\chi \in \widehat{H} \\ \chi \neq \textup{id}}}L(E/L_n, \chi,s).\]  
We let $n$ be as in part (1) and $\chi$ a non-trivial character of $H$. Then, $L(E/L_n, \chi,s)=L(E/\mathbb{Q}, \textup{Ind}_{H}^{G}\chi,1) \neq 0$. Therefore, $L(E/\mathbb{Q}, \tau,1) \neq 0$ for all primitive characters $\tau \in \widehat{G}$. Then, by \cite[Corollary 14.3]{Kato}, the $\tau$-part $(\Sha_{n+1})^{(\tau)}$ is finite for all primitive $\tau \in \widehat{G}$. This completes the proof of (2). For (3), it suffices to show that $R(L_n)$ is finite. This follows in the same way as (2).  \qedhere
\end{proof}

\begin{theorem} \label{size_of_Sha}
Let $p \geq 11$ be a prime. Then, there exists an abelian extension $L/\mathbb{Q}$ and an abelian variety $R/L$ such that $\left | \Sha(R/L) \right | = p \cdot \square.$ 
\end{theorem}
\begin{proof}
We let $E/\mathbb{Q}$ be an elliptic curve defined over $\mathbb{Q}$ and $d$ a positive integer such that $p$ and all bad primes of $E/\mathbb{Q}$ split in $\mathbb{Q}(\sqrt{-d})$.
For the remainder of this proof, we fix $n$ as in Theorem \ref{RK} so that $\textup{rk}_{p}(E/L_{n+1})=\textup{rk}_{p}(E/L_n)$ and $\Sha(R/L_n)$ is finite. In view of Lemma \ref{Lem:size_of_shas}(1)-(2), we deduce that $\left | \Sha(R/L_n) \right | \equiv \left | \Sha_0(R/L_n)[p^{\infty}] \right |$. Combining this with Proposition \ref{Prop:CT_Weil_Restriction}(2) (applied with $K=K_n, F=K_{n+1},L=L_n$) and Lemma \ref{torsion_quotient_elliptic_curve}, we get
\[ \label{cas_tate_ec} \left | \Sha(R/L_{n})\right | \equiv p^{\textup{rk}_{p}(E/L_{n})}    \Mod \square. \tag{$\dagger$}\]
Since $E$ is defined over $\mathbb{Q}$, its global root number $\omega(E/\mathbb{Q}(\sqrt{-d}))=-1$. This is because the only contribution in the root number computation comes from the unique archimedean place of $\mathbb{Q}(\sqrt{-d})$. Since $\textup{Res}_{\mathbb{Q}}^{\mathbb{Q}(\sqrt{-d})} E$ is isogenous to $E \times E_{-d}$, we deduce that
\fontsize{10.5pt}{9pt}
\[ (-1)^{\textup{rk}_{p}(E/\mathbb{Q}(\sqrt{-d}))} = (-1)^{\textup{rk}_{p}(E/\mathbb{Q})} (-1)^{\textup{rk}_{p}(E_{-d}/\mathbb{Q})} \stackrel{\textup{\cite[Thm 1.4]{BSD-mod-squares}}}= w(E/\mathbb{Q})w(E_{-d}/\mathbb{Q}). \]
In view of the isogeny, $ w(E/\mathbb{Q}(\sqrt{-d}))= w(E/\mathbb{Q})w(E_{-d}/\mathbb{Q})$, from which we deduce that $\textup{rk}_{p}(E/\mathbb{Q}(\sqrt{-d}))  \equiv 1 \Mod 2$. Since the parity of $p^{\infty}$-Selmer ranks remains unchanged in odd degree Galois extensions \cite[Corollary 4.15]{BSD-mod-squares}, then $\textup{rk}_{p}(E/L_{n}) \equiv \textup{rk}_{p}(E/\mathbb{Q}(\sqrt{-d})) \equiv 1 \Mod 2$. Combining this with \eqref{cas_tate_ec} gives the required result. \qedhere
\end{proof}

\begin{theorem} \label{Sha-p-square}Let $p$ be a rational prime. Then, there exists an abelian variety $R'/\mathbb{Q}$ such that $\left | \Sha(R'/\mathbb{Q})\right |=p\cdot \square.$
\end{theorem}
\begin{proof}
For $p\geq 11$, we consider $R' = \textup{Res}^{L}_{\mathbb{Q}}R$ where $R$ is the abelian variety from Theorem \ref{size_of_Sha}. Since $\Sha(R'/\mathbb{Q}) \cong \Sha(R/L)$ (combine Shapiro's lemma with \cite[pp. 178(a)]{MR0330174}), the result follows by Theorem \ref{size_of_Sha}. For $p\leq 7$, see \cite{poonen1999cassels} and \cite{Stein}.
\end{proof}

\begin{corollary} \label{Sha-n-square}
For every positive square-free integer $n$, there exists an abelian variety defined over $\mathbb{Q}$ with finite Tate--Shafarevich group of order $n \cdot \square$.
\end{corollary} 
\begin{proof}
Given abelian varieties $A_1, A_2$, $\Sha(A_1\times A_2) \cong \Sha(A_1) \times \Sha(A_2)$. The result follows from Theorem \ref{Sha-p-square}.
\end{proof}

\end{document}